\documentclass[reqno]{amsart}
\usepackage[top=25truemm,bottom=25truemm,left=25truemm,right=25truemm]{geometry}
\usepackage{bm}
\usepackage{amsmath,amssymb}
\usepackage{mathrsfs}
\usepackage{enumerate}
\usepackage{multirow}
\usepackage{cases}
\usepackage{color}

\allowdisplaybreaks

\newcommand{\ds}{\displaystyle}
\newcommand{\nn}{\nonumber}
\newcommand{\ra}{\rightarrow}

\newcommand{\ex}{\exists}

\newcommand{\st}{\mbox{~s.t.~}}
\newcommand{\leftright}[3]{\left#1#2\right#3}

\newcommand{\BrS}[1]{\leftright{(}{#1}{)}}
\newcommand{\BrM}[1]{\left\{#1\right\}}

\newcommand{\BrA}[1]{\left|#1\right|}

\newcommand{\ol}[1]{\overline{#1}}

\newcommand{\Diag}{\textrm{diag}}

\newcommand{\Dim}{d}
\newcommand{\midd}{\;\middle|\;}

\newcommand{\Set}[2]{\left\{#1\midd#2\right\}}
\newcommand{\SetNd}[2]{\{#1\mid#2\}}

\newcommand{\Meas}[1]{\BrA{#1}} 
\newcommand{\Norm}[2]{\left\|#1\right\|_{#2}} 

\newcommand{\Grad}{\nabla}
\newcommand{\Lap}{\Delta}
\newcommand{\Div}{\nabla\cdot}

\newcommand{\Dm}{\varOmega}
\newcommand{\DmOl}{\ol{\Dm}}
\renewcommand{\H}{H}
\newcommand{\DmH}{\Dm_\H}

\newcommand{\Bd}{\Gamma}
\newcommand{\BdH}{\Bd_{\H}}

\newcommand{\N}{N}
\newcommand{\PtSetCher}{\mathcal{X}}
\newcommand{\PtSet}{\PtSetCher_{\N}}

\newcommand{\IndexSet}[1]{\Lambda(#1)}

\renewcommand{\i}{i}
\renewcommand{\j}{j}
\renewcommand{\k}{k}
\newcommand{\PtChar}{x}
\newcommand{\Pt}[2]{\PtChar_{#1}^{\TimeIndex{#2}}}

\newcommand{\Dx}{\Delta\x}

\newcommand{\h}{h}

\newcommand{\WeightCher}{w}
\newcommand{\Weight}{\WeightCher}

\newcommand{\WeightH}{\Weight_{\h}}

\newcommand{\w}{\Weight}

\newcommand{\wh}{\WeightH}

\newcommand{\FsSolCher}{V}
\newcommand{\FsSol}{\FsSolCher}
\newcommand{\FsSolApp}{\FsSolCher_{\h}}
\newcommand{\FuncCher}{v}

\newcommand{\Func}{\FuncCher}

\newcommand{\PvChar}{\omega}
\newcommand{\Pv}[1]{\PvChar_{#1}}
\newcommand{\Pvi}{\Pv{\i}}

\newcommand{\PvSetCher}{\mathcal{V}}
\newcommand{\PvSet}{\PvSetCher_{N}}

\newcommand{\GradApp}[1]{\Grad_\h^{\TimeIndex{#1}}}
\newcommand{\LapApp}[1]{\Delta_\h^{\TimeIndex{#1}}}

\newcommand{\DivAppPlus}{\Grad_\h^{+}\cdot}

\newcommand{\FsWeightFunc}{\mathcal{W}}

\newcommand{\PvTempCher}{\widetilde{\PvChar}}
\newcommand{\PvTemp}[1]{\PvTempCher_{\i}}

\newcommand{\FsCChar}{C}
\newcommand{\FsCz}[1]{\FsCChar(#1)}
\newcommand{\FsC}[2]{\FsCChar^{#1}(#2)}

\newcommand{\TimeIndex}[1]{#1}

\newcommand{\mderivD}{{\rm D}}
\newcommand{\mderiv}[3]{
	\ifnum #1=1	\frac{\mderivD#3}{\mderivD{#2}}%
	\else{\frac{\mderivD^{#1}#3}{\mderivD^{#1}{#2}} %
	}\fi%
}

\newcommand{\dN}{\mathbb{N}}

\newcommand{\dR}{\mathbb{R}}
\newcommand{\dRd}{\mathbb{R}^{\Dim}}

\newcommand{\xchar}{x}
\newcommand{\x}{\xchar}

\newcommand{\ychar}{y}
\newcommand{\y}{\ychar}

\newcommand{\dx}{d\x}

\def\deq{\mathrel{\mathop:}=}%

\newcommand{\Sol}{u}

\newcommand{\DmFuncCher}{f}
\newcommand{\DmFunc}{\DmFuncCher}
\newcommand{\DmFuncEx}{\widehat{\DmFuncCher}}

\newcommand{\SolApp}{\mathrm{\Sol}}

\newcommand{\Np}{\N_{\Dm}}

\newcommand{\IndexSetSeq}[2]{{#1}_{#2}}
\newcommand{\IndexSetSeqLast}{K}

\newcommand{\InnerProdDisc}[3]{(#1,#2)_{h(#3)}}
\newcommand{\NormDiscL}[3]{\Norm{#1}{\ell^{#2}(#3)}}

\newcommand{\NormDiscHz}[3]{\Norm{#1}{h^{#2}_0(#3)}}

\usepackage{amsmath,amssymb}
\newtheorem{theorem}{Theorem}
\newtheorem{definition}[theorem]{Definition}
\newtheorem{lemma}[theorem]{Lemma}
\newtheorem{remark}[theorem]{Remark}

\usepackage{comment}
\allowdisplaybreaks

\usepackage{mathtools}
\mathtoolsset{showonlyrefs=true}

\newcommand{\ConstSemiReg}{c_0}
\newcommand{\Matrix}{A}
\newcommand{\MatrixDiag}{D}
\newcommand{\VecSol}{y}
\newcommand{\VecSouce}{b}
\newcommand{\MatrixComp}[2]{a_{#1#2}}
\newcommand{\CofPositiveDifinite}{\alpha}
\newcommand{\InteractCofGrad}[2]{I_{#1 #2}}
\newcommand{\InteractCofLap}[2]{J_{#1 #2}}
\newcommand{\ConstThm}{c}
\newcommand{\IndexSetLatticeNext}[1]{\lambda_{#1}}
\newcommand{\DmLattice}[1]{\sigma_{#1}}
\newcommand{\DmLatticeOl}[1]{\ol{\sigma}_{#1}}
\newcommand{\BdLattice}[2]{\gamma_{#1#2}}

\title[Stability of a generalized particle method for a Poisson equation by discrete Sobolev norms]{Stability of a generalized particle method for a Poisson equation by \\discrete Sobolev norms}

\author[Y. Imoto]{Yusuke Imoto${}^\dagger$}
\address{${}^\dagger$Tohoku Forum for Creativity, Tohoku University, 2-1-1 Katahira, Aoba-ku, Sendai 980-8577 Japan}
\email{y-imoto@tohoku.ac.jp}


\begin{document}
	
	\begin{abstract}
		Numerical analysis is conducted for a generalized particle method for a Poisson equation. 
		Unique solvability is derived for the discretized Poisson equation by introducing a connectivity condition for particle distributions. 
		Moreover, by introducing discrete Sobolev norms and a semi-regularity of a family of discrete parameters, stability is obtained for the discretized Poisson equation based on the norms.  
	\end{abstract}

\maketitle

\section{Introduction}
Numerical analysis is conducted for the generalized particle method introduced in \cite{imoto2016phd}. 
This method is a generalization of conventional particle methods such as smoothed particle hydrodynamics \cite{liu2010smoothed} and the moving particle semi-implicit \cite{koshizuka1996moving}. 
For the generalized particle method, we have established numerical analyses involving the truncation error estimates of approximate operators \cite{imoto2016teintpV,imoto2017tederiV} and the error estimates for the Poisson and heat equations based on the maximum norm \cite{imoto2016phd}. 
Therefore, as the next step of this study, we focus on the numerical analyses of the generalized particle method using discrete Sobolev norms.

This study considers a Poisson equation with a source term given by a divergence form. 
This formulation is selected because it has several practical applications. 
For example, a pressure Poisson equation, which is appeared in formulations of particle methods for the incompressible Navier--Stokes equations \cite{shao2003incompressible}, uses a source term including a divergence of a predictor of velocity. 

A connectivity condition for particle distributions and a semi-regularity of a family of discrete parameters are introduced for analyzing the discretized Poisson equation. 
By virtue of the connectivity condition, a unique solvability of the discretized Poisson equation is derived. 
Moreover, by demonstrating certain properties of the discrete Sobolev norms, such as integration by parts, the stability of the discretized Poisson equation is obtained with semi-regularity.   

\section{Formulation}
\label{sec:formulation}
Let $\Dm\subset\dRd\,(\Dim\in\dN)$ be a bounded domain with smooth boundary $\Bd$. 
Let $\FsCz{\DmOl}$ be the space of real continuous functions defined in $\DmOl$.
For $k\in\dN$, let $\FsC{k}{\DmOl}$ be the space of functions in $\FsCz{\DmOl}$ with derivatives up to the $k$th order. 
We define a function space, $\FsSol$, as
\begin{equation}
	\FsSol\deq\Set{\Func\in\FsCz{\DmOl}}{\Func(\x)=0\,(\x\in\Bd)}.
\end{equation}
Then, we consider the following Poisson equation with a homogeneous boundary condition:
\begin{equation}
	\mbox{Find}~\Sol\in\FsSol\cap\FsC{2}{\DmOl} \st -\Lap\Sol = \Div\DmFunc,
	\label{Poisson}
\end{equation}
where $\DmFunc\in(\FsSol\cap\FsC{1}{\DmOl})^\Dim$ is given. 

We introduce approximate operators in the generalized particle method. 
Let $\H$ be a fixed positive number. 
For $\Dm$ and $\H$, we define $\DmH\subset\dRd$ and $\BdH$ by 
\begin{align}
	\DmH &\deq \left\{\x\in\dRd \midd \ex \y\in\Dm \st |\x-\y|<\H\right\}, \\
	\BdH &\deq \DmH\setminus\Dm. 
\end{align}
For $\N\in\dN$, we define a particle distribution, $\PtSet$, and a particle volume set, $\PvSet$, as
\begin{align}
	\PtSet \deq \Set{\Pt{\i}{}\in\Dm}{\i=1,2,\dots,\N,\,\Pt{\i}{}\neq\Pt{\j}{}\,(\i\neq\j)}, 
	\label{def:PtSet}\\
	\PvSet \deq \Set{\Pv{\i}>0}{\i=1,2,\dots,\N,\,\sum_{\i=1}^\N\Pv{\i}=\Meas{\DmH}}. 
	\label{def:PvSet}
\end{align}
Here, $\Meas{\DmH}$ denotes the volume of $\DmH$. 
We refer to $\Pt{\i}{}\in\PtSet$ and $\Pv{\i}\in\PvSet$ as a particle and particle volume, respectively.   
We define a function set, $\FsWeightFunc$, as
\begin{equation}
	\FsWeightFunc \deq \Set{\w:[0,\infty)\ra\dR}{\w(r)>0\,(0<r<1),\w(r)=0\,(r\geq1),\,\int_{\dRd}\w(|\x|)\dx=1}. 
\end{equation}
We refer to $\w\in\FsWeightFunc$ as a reference weight function. 
We define an influence radius, $\h$, as a positive number that satisfies $\min\SetNd{|\Pt{\i}{}-\Pt{\j}{}|}{\i\neq\j}<\h<\H$. 
For reference weight function $\w\in\FsWeightFunc$ and influence radius $\h$, we define a weight function, $\wh:[0,\infty)\ra\dR$, as
\begin{equation}
	\wh(r) \deq \frac{1}{\h^{\Dim}}\w\left(\frac{r}{\h}\right). 
	\label{def:wh}
\end{equation}
Then, for discrete parameters $(\PtSet, \PvSet, \h)$ and reference weight function $\w\in\FsWeightFunc$, we define an approximate divergence operator, $\DivAppPlus$, for $\psi:\PtSet\ra\dRd$ as
\begin{equation}
	\DivAppPlus\psi_{\i} \deq \Dim\sum_{\j\neq\i}\Pv{\j} \frac{\psi_{\j}+\psi_{\i}}{|\Pt{\j}{}-\Pt{\i}{}|}\cdot\frac{\Pt{\j}{}-\Pt{\i}{}}{|\Pt{\j}{}-\Pt{\i}{}|} \wh(|\Pt{\j}{}-\Pt{\i}{}|)
	\label{def:DivAppPlus}
\end{equation}
and an approximate Laplace operator, $\LapApp{}$, for $\phi:\PtSet\ra\dR$ as
\begin{equation}
	\LapApp{} \phi_{\i}\deq 2\Dim\sum_{\j\neq\i}\Pv{\i} \frac{\phi_{\j}-\phi_{\i}}{|\Pt{\j}{}-\Pt{\i}{}|^2} \wh(|\Pt{\j}{}-\Pt{\i}{}|), 
	\label{def:LapApp}
\end{equation}
where $\psi_{\i}\deq\psi(\Pt{\i}{})$ and  $\phi_{\i}\deq\phi(\Pt{\i}{})$. 

We define an index set, $\IndexSet{S}\,(S\subset\dRd)$, and a function space, $\FsSolApp$, as 
\begin{align}
	&\IndexSet{S}\deq \BrM{\i=1,2,\dots,\N\midd \Pt{\i}{}\in \PtSet\cap S},\\
	&\FsSolApp\deq\Set{\Func:\PtSet\ra\dR}{\Func(\Pt{\i}{})=0\,(\i\in\IndexSet{\BdH})}. 
\end{align}
Then, we consider the following generalized particle method for the Poisson equation: 
\begin{equation}
	\mbox{Find}~\SolApp\in\FsSolApp \st -\LapApp{}\SolApp_{\i} = \DivAppPlus\DmFuncEx_{\i}\quad\i\in\IndexSet{\Dm}, 
	\label{Poisson_disc}
\end{equation}
where $\DmFuncEx\in\FsSolApp^\Dim$ such that $\DmFuncEx_{\i}=\DmFunc_{\i}\,(\i\in\IndexSet{\Dm})$, $\DmFuncEx_{\i}=0\,(\i\in\IndexSet{\BdH})$. 

\begin{remark}
	In particle methods, approximate operators of the first derivative, including the difference in function values, have also been proposed, such as in \cite{liu2010smoothed}. 
	To distinguish between our operators and their operators, we used the notation with the plus symbol, e.g., $\DivAppPlus$ in \eqref{def:DivAppPlus}.
\end{remark}

\begin{remark}
We can introduce approximate operators \eqref{def:DivAppPlus} and \eqref{def:LapApp} using the weighted averages of approximations based on the finite volume method, as shown in Appendix \ref{appendix:deriv_app_diver}. 
\end{remark}

\section{Connectivity and semi-regularity}
We introduce a connectivity condition for particle distribution $\PtSet$ and a semi-regular condition for a family of discrete parameters. 
\begin{definition}
\label{def_connectivity}
For influence radius $\h$, we consider that particle distribution $\PtSet$ satisfies $\h$-connectivity if for all $\i\in\IndexSet{\Dm}$, there exists an integer, $\IndexSetSeqLast\,(\leq\N)$, and a sequence, $\{\IndexSetSeq{\i}{\k}\}_{\k=1}^{\IndexSetSeqLast}\subset\{1,2,\dots,\N\}$, such that
\begin{equation}
	\IndexSetSeq{\i}{1}=\i,\quad|\Pt{\IndexSetSeq{\i}{\k}}{}-\Pt{\IndexSetSeq{\i}{\k+1}}{}|<\h\,(1\leq\k<\IndexSetSeqLast),\quad\IndexSetSeq{\i}{\k}\in\IndexSet{\Dm}\,(1\leq\k<\IndexSetSeqLast),\quad\IndexSetSeq{\i}{\IndexSetSeqLast}\in\IndexSet{\BdH}. 
	\label{h_connectivity}
\end{equation}
\end{definition}

\begin{definition}
	A family, $\{(\PtSet, \PvSet, \h)\}$, is \textit{semi-regular} if there exists a positive constant, $\ConstSemiReg$, such that for all elements of the family,
 \begin{equation}
  \max_{\i=1,2,\dots,\N}\BrM{\sum_{\j\neq\i}\Pv{\j}\wh(|\Pt{\j}{}-\Pt{\i}{}|)}\leq \ConstSemiReg. 
  \label{def:semi-regular}
 \end{equation}
 We refer to the constant, $\ConstSemiReg$, as a semi-regular constant. 
  \end{definition}

\begin{remark}
Consider graph $G$, whose vertex set is particle distribution $\PtSet$ and whose edges are a pair, $(\Pt{\i}{}, \Pt{\j}{})$, which satisfies $0<|\Pt{\j}{}-\Pt{\i}{}|<\h$, e.g., see Figure \ref{fig:connectivity}. 
By Definition \ref{def_connectivity}, particle distribution $\PtSet$ satisfies $\h$-connectivity is equivalent to all vertices of $G$ on $\Dm$ have a path to a vertex of $G$ on $\BdH$. 
\end{remark}

\begin{figure}[t]
\centering
\includegraphics[width=100mm,bb=0 0 1397.7mm 418mm]{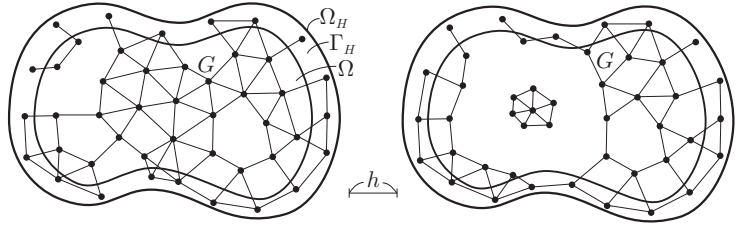}
\caption{Particle distributions $\PtSet$ and their graphs $G$: the left and right sides show the particle distributions with and without $\h$-connectivity, respectively.}
\label{fig:connectivity}
\end{figure}

\section{Stability analysis}

First, we show the unique solvability for the discrete Poisson equation \eqref{Poisson_disc}. 
\begin{theorem}
\label{theorem:unique_solvability}
If particle distribution $\PtSet$ satisfies $\h$-connectivity, then discrete Poisson equation \eqref{Poisson_disc} has a unique solution. 
\end{theorem}

\begin{proof}
Let $\Np$ be the number of particles included in $\Dm$. 
We renumber the index of particles so that $\Pt{\i}{}\in\Dm\,(\i=1,2,\dots,\Np)$. 
Let $\MatrixComp{\i}{\j}\,(i,j=1,2,\dots,N)$ be
\begin{align}
\MatrixComp{\i}{\j}&\deq 
\begin{cases}
\ds 0, \quad & i=j,\\
\ds 2 \Dim \frac{\wh(|\Pt{\j}{}-\Pt{\i}{}|)}{|\Pt{\j}{}-\Pt{\i}{}|^2}, \quad & \i\neq\j. 
\end{cases}
\end{align}
We define matrix $\Matrix\in\dR^{\Np\times\Np}$ as
\begin{equation}
	\Matrix_{\i\j}\deq 
	\begin{cases}
		\ds \sum_{k=1}^{\N} \frac{\Pv{k}}{\Pv{\i}}\MatrixComp{\i}{\k}, \quad & i=j,\\
		\ds -\MatrixComp{\i}{\j}, \quad & \i\neq\j,
	\end{cases}
\end{equation}
Then, discrete Poisson equation \eqref{Poisson_disc} is equivalent to
\begin{equation}
	\mbox{Find}~\VecSol\in\dR^{\Np} \st \Matrix\MatrixDiag\VecSol=\VecSouce, 
\end{equation}
where $\MatrixDiag\deq \Diag(\Pv{\i})$, $\VecSouce_{\i}\deq\DivAppPlus\DmFuncEx_{\i}\,(\i=1,2,\dots,\Np)$, and $\VecSol_{\i}\deq \SolApp_{\i}\,(\i=1,2,\dots,\Np)$. 
As $\Pv{\i}>0\,(\i=1,2,\dots,\N)$, diagonal matrix $\MatrixDiag$ is a regular matrix. 
Therefore, it is sufficient to prove that $\Matrix$ is a regular matrix. 
As $\Matrix$ is symmetric, we will prove that $\Matrix$ is a positive definite matrix. 
For all $\CofPositiveDifinite\in\dR^{\Np}\setminus\{0\}$, we have
\begin{align}
	\sum_{i,j=1}^{\Np} \CofPositiveDifinite_{\i}\CofPositiveDifinite_{\j} \Matrix_{\i\j}
		&= 2\sum_{1\leq i<j\leq\Np} \CofPositiveDifinite_{\i}\CofPositiveDifinite_{\j} \Matrix_{\i\j} + \sum_{\i=1}^{\Np} \CofPositiveDifinite_{\i}^2 \Matrix_{\i\i}
	\nn\\
	& = -2\sum_{1\leq i<j\leq\Np}\CofPositiveDifinite_{\i}\CofPositiveDifinite_{\j} \MatrixComp{\i}{\j} + \sum_{\i=1}^{\Np} \CofPositiveDifinite_{\i}^2 \sum_{k=1}^{\N} \frac{\Pv{k}}{\Pv{\i}}\MatrixComp{\i}{\k}
	\nn\\
	& = \sum_{1\leq i<j\leq\Np}\frac{\BrS{\Pv{\j} \CofPositiveDifinite_{\i}-\Pv{\i}\CofPositiveDifinite_{\j}}^2}{\Pv{\i} \Pv{\j}}\MatrixComp{\i}{\j} + \sum_{\i=1}^{\Np} \CofPositiveDifinite_{\i}^2 \sum_{k=\Np+1}^{\N} \frac{\Pv{k}}{\Pv{\i}}\MatrixComp{\i}{\k}. 
	\label{prf_poisson_unique_01}
\end{align}
As $\MatrixComp{\i}{\j}$ is nonnegative, \eqref{prf_poisson_unique_01} is nonnegative. 
For $a\in\dR^{\Np}\setminus\{0\}$, we set $\i$ such that $\CofPositiveDifinite_{\i}\neq 0$. 
Because of particle distribution $\PtSet$ with $\h$-connectivity, we can consider a sequence, $\{\IndexSetSeq{\i}{\k}\}_{\k=1}^{\IndexSetSeqLast}$, such that \eqref{h_connectivity}. 
As all terms of the last equation in \eqref{prf_poisson_unique_01} are nonnegative, we have 
\begin{align}
	\sum_{i,j=1}^{\Np} \CofPositiveDifinite_{\i}\CofPositiveDifinite_{\j} \Matrix_{\i\j}
	& \geq \sum_{k=1}^{\IndexSetSeqLast-1}\frac{\BrS{\Pv{\IndexSetSeq{\i}{\k+1}} \CofPositiveDifinite_{\IndexSetSeq{\i}{\k}}-\Pv{\IndexSetSeq{\i}{\k}}\CofPositiveDifinite_{\IndexSetSeq{\i}{\k+1}}}^2}{\Pv{\IndexSetSeq{\i}{\k}}\Pv{\IndexSetSeq{\i}{\k+1}}}\MatrixComp{\IndexSetSeq{\i}{\k}}{\IndexSetSeq{\i}{\k+1}}+\frac{\Pv{\IndexSetSeq{\i}{\IndexSetSeqLast}} }{\Pv{\IndexSetSeq{\i}{\IndexSetSeqLast-1}}}\CofPositiveDifinite_{\IndexSetSeq{\i}{\IndexSetSeqLast}}^2 \MatrixComp{\IndexSetSeq{\i}{\IndexSetSeqLast-1}}{ \IndexSetSeq{\i}{\IndexSetSeqLast}}. 
	\label{proof:unieque:positive_definite_subseq}
\end{align}
As $|\Pt{\IndexSetSeq{\i}{\k}}{}-\Pt{\IndexSetSeq{\i}{\k+1}}{}|<\h$, the value of $\MatrixComp{\IndexSetSeq{\i}{\k}}{\IndexSetSeq{\i}{\k+1}}\,(\k=1,2,\dots,\IndexSetSeqLast-1)$ is positive. 
Thus, if the right hand side of \eqref{proof:unieque:positive_definite_subseq} equals zero, then $\CofPositiveDifinite_{\IndexSetSeq{\i}{\k}}=0\,(\k=1,2,\dots,\IndexSetSeqLast)$. 
As this is inconsistent with $\CofPositiveDifinite_{\i}\,(=\CofPositiveDifinite_{\IndexSetSeq{\i}{1}})\neq 0$, the right hand side of \eqref{proof:unieque:positive_definite_subseq} is positive. 
Therefore, matrix $\Matrix$ is a positive definite matrix. 
\end{proof}

Next, we introduce a few notations and show certain lemmas. 
Hereafter, assume that particle distribution $\PtSet$ satisfies $\h$-connectivity.  
For $S\subset\dRd$ and $n\in\dN$, we define a discrete inner product, $\InnerProdDisc{\cdot}{\cdot}{S}:\FsSolApp^n\times\FsSolApp^n\ra\dR$, a discrete $L^2$ norm, $\NormDiscL{\cdot}{2}{S}:\FsSolApp^n\ra\dR$, and a discrete $H^1_0$ norm, $\NormDiscHz{\cdot}{1}{S}:\FsSolApp^n\ra\dR$, as
\begin{align}
	&\InnerProdDisc{\phi}{\psi}{S} \deq \sum_{\i\in\IndexSet{S}}\Pvi\,\phi_{\i}\cdot\psi_{\i},\\
	&\NormDiscL{\phi}{2}{S} \deq \InnerProdDisc{\phi}{\phi}{S}^{1/2}=\BrS{\sum_{\i\in\IndexSet{S}}\Pvi\,\phi_{\i}^2}^{1/2},\\
	&\NormDiscHz{\phi}{1}{S} \deq \BrS{\Dim \sum_{\i\in\IndexSet{S}}\Pv{\i} \sum_{\j\neq\i}\Pv{\j} \frac{|\phi_{\j}-\phi_{\i}|^2}{|\Pt{\j}{}-\Pt{\i}{}|^2}\wh(|\Pt{\j}{}-\Pt{\i}{}|)}^{1/2}.
\end{align}
For $\phi:\PtSet\ra\dR$, we define an approximate gradient operator, $\GradApp{}$, by
\begin{equation}
	\GradApp{} \phi_{\i} \deq \Dim\sum_{\j\neq\i}\Pv{\j}\frac{\phi_{\j}-\phi_{\i}}{|\Pt{\j}{}-\Pt{\i}{}|}\frac{\Pt{\j}{}-\Pt{\i}{}}{|\Pt{\j}{}-\Pt{\i}{}|}\wh(|\Pt{\j}{}-\Pt{\i}{}|)
	\label{def:GradApp}
\end{equation}
Then, we obtain the following lemma: 

\begin{lemma}
	\label{lem:propaties_disc_norm}
	For $\phi\in\FsSolApp$ and $\psi\in\FsSolApp^\Dim$, we have 
	\begin{equation}
		\InnerProdDisc{\DivAppPlus\psi}{\phi}{\Dm}
			=-\InnerProdDisc{\psi}{\GradApp{}\phi}{\Dm},
			\label{lem:propaties_disc_norm:integral_of_parts}
	\end{equation}
	\begin{equation}
		-\InnerProdDisc{\LapApp{}\phi}{\phi}{\Dm}=\NormDiscHz{\phi}{1}{\DmH}^2\geq\NormDiscHz{\phi}{1}{\Dm}^2.
			\label{lem:propaties_disc_norm:equiv_norm}
	\end{equation}
\end{lemma}

\begin{proof}
	First, we prove \eqref{lem:propaties_disc_norm:integral_of_parts}. 
	Let $\InteractCofGrad{\i}{\j}$ be
	\begin{equation}
		\InteractCofGrad{\i}{\j}\deq
		\begin{cases}
			0,\quad&\i=\j,\\
			\ds\Dim\frac{\Pt{\j}{}-\Pt{\i}{}}{|\Pt{\j}{}-\Pt{\i}{}|^2} \wh(|\Pt{\j}{}-\Pt{\i}{}|),\quad&\i\neq\j. 
		\end{cases}
		\label{def:InteractCofGrad}
	\end{equation}
	As $\phi\in\FsSolApp$, $\psi\in\FsSolApp^\Dim$, and $\InteractCofGrad{\i}{\j}=-\InteractCofGrad{\j}{\i}$, we have
	\begin{align}
		\InnerProdDisc{\DivAppPlus\psi}{\phi}{\Dm} 
			&=\sum_{\i\in\IndexSet{\Dm}}\Pv{\i}\phi_{\i}\sum_{\j=1}^\N\Pv{\j}\BrS{\psi_{\j}+\psi_{\i}}\cdot\InteractCofGrad{\i}{\j}\\
			&=\sum_{\i=1}^\N\sum_{\j=1}^\N\Pv{\i}\Pv{\j}\phi_{\i}\BrS{\psi_{\j}+\psi_{\i}}\cdot\InteractCofGrad{\i}{\j}\\
			&=\dfrac{1}{2}\sum_{\i=1}^\N\sum_{\j=1}^\N\Pv{\i}\Pv{\j}(\phi_{\i}-\phi_{\j})\BrS{\psi_{\j}+\psi_{\i}}\cdot\InteractCofGrad{\i}{\j}\\
			&=\sum_{\i=1}^\N\Pv{\i}\psi_{\i}\cdot\sum_{\j=1}^\N\Pv{\j}(\phi_{\i}-\phi_{\j})\InteractCofGrad{\i}{\j}\\
			&=-\InnerProdDisc{\psi}{\GradApp{}\phi}{\Dm}. 
	\end{align}
	Next, we prove \eqref{lem:propaties_disc_norm:equiv_norm}. 
	Let $\InteractCofLap{\i}{\j}\,(\geq0)$ be
	\begin{equation}
		\InteractCofLap{\i}{\j}\deq
		\begin{cases}
			0,\quad&\i=\j,\\
			\ds\Dim\frac{\wh(|\Pt{\j}{}-\Pt{\i}{}|)}{|\Pt{\j}{}-\Pt{\i}{}|^2} ,\quad&\i\neq\j. 
		\end{cases}
	\end{equation}
	As $\phi\in\FsSolApp$ and $\InteractCofLap{\i}{\j}=\InteractCofLap{\j}{\i}$, we have
	\begin{align}
		-\InnerProdDisc{\LapApp{}\phi}{\phi}{\Dm}
			&=2\sum_{\i\in\IndexSet{\Dm}}\Pv{\i}\phi_{\i}\sum_{\j=1}^\N\Pv{\j}\BrS{\phi_{\i}-\phi_{\j}}\InteractCofLap{\i}{\j}\\
			&=2\sum_{\i=1}^\N\sum_{\j=1}^\N\Pv{\i}\Pv{\j}\phi_{\i}\BrS{\phi_{\i}-\phi_{\j}}\InteractCofLap{\i}{\j}\\
			&=\sum_{\i=1}^\N\sum_{\j=1}^\N\Pv{\i}\Pv{\j}\BrS{\phi_{\i}-\phi_{\j}}^2\InteractCofLap{\i}{\j}\\
			&=\NormDiscHz{\phi}{1}{\DmH}^2\\
	&=\NormDiscHz{\phi}{1}{\Dm}^2 +
				\sum_{\i\in\IndexSet{\BdH}}\Pv{\i} \sum_{\j=1}^\N\Pv{\j}\phi_{\j}^2\InteractCofLap{\i}{\j}\\
			&\geq\NormDiscHz{\phi}{1}{\Dm}^2. 
	\end{align}
\end{proof}

\begin{lemma}
	\label{lem:inequarity_disc_norm}
	Assume that a family, $\{(\PtSet, \PvSet, \h)\}$, is semi-regular. 
	Then, we have
	\begin{equation}
		\NormDiscL{\GradApp{}\phi}{2}{\Dm}^2
		\leq \Dim\,\ConstSemiReg\NormDiscHz{\phi}{1}{\Dm}^2. 
	\end{equation}
	Here, $\ConstSemiReg$ is the semi-regular constant in \eqref{def:semi-regular}. 
\end{lemma}

\begin{proof}
	By the Cauchy--Schwarz inequality, we have 
		\begin{align}
			\NormDiscL{\GradApp{}\phi}{2}{\Dm}^2&= \sum_{\i\in\IndexSet{\Dm}}\Pvi\BrA{\GradApp{}\phi_{\i}}^2\\
			&\leq \Dim^2\sum_{\i\in\IndexSet{\Dm}}\Pvi\BrS{\sum_{\j\neq\i}\Pv{\j}\frac{|\phi_{\j}-\phi_{\i}|}{|\Pt{\j}{}-\Pt{\i}{}|}\wh(|\Pt{\j}{}-\Pt{\i}{}|)}^2\\
			&\leq \Dim^2\sum_{\i\in\IndexSet{\Dm}}\Pvi\BrS{\sum_{\j\neq\i}\Pv{\j}\frac{|\phi_{\j}-\phi_{\i}|^2}{|\Pt{\j}{}-\Pt{\i}{}|^2}\wh(|\Pt{\j}{}-\Pt{\i}{}|)}\BrS{\sum_{\j\neq\i}\Pv{\j}\wh(|\Pt{\j}{}-\Pt{\i}{}|)}. 
	\end{align}
		As family $\{(\PtSet, \PvSet, \h)\}$ is semi-regular, we obtain
	\begin{equation}
		\NormDiscL{\GradApp{}\phi}{2}{\Dm}^2\leq\Dim\,\ConstSemiReg\NormDiscHz{\phi}{1}{\Dm}^2. 
	\end{equation}
\end{proof}

Then, we establish the following stability of the generalized particle method for Poisson equation \eqref{Poisson_disc}. 
\begin{theorem}
\label{thm:stability}
Assume that particle distribution $\PtSet$ satisfies $\h$-connectivity and family $\{(\PtSet, \PvSet, \h)\}$ is semi-regular. 
Then, there exists constant $c$ such that
\begin{equation}
	\NormDiscHz{\SolApp}{1}{\Dm}\leq \ConstThm\NormDiscL{\DmFunc}{2}{\Dm}. 
	\label{thm:stability:ineq}
\end{equation}
\end{theorem}
\begin{proof}
By the Cauchy--Schwarz inequality, \eqref{Poisson_disc}, and Lemmas \ref{lem:propaties_disc_norm} and \ref{lem:inequarity_disc_norm}, we have
\begin{align*}
	\NormDiscHz{\SolApp}{1}{\Dm}^2
		&\leq\BrA{-\InnerProdDisc{\LapApp{}\SolApp}{\SolApp}{\Dm}}
		\\
		&=\BrA{\InnerProdDisc{\DivAppPlus\widehat{f}}{\SolApp}{\Dm}}
		\\
		&=\BrA{-\InnerProdDisc{\widehat{f}}{\GradApp{}\SolApp}{\Dm}}
		\\
		&\leq \NormDiscL{\DmFunc}{2}{\Dm}\NormDiscL{\GradApp{}\SolApp}{2}{\Dm}
		\\
		&\leq \sqrt{\Dim\,\ConstSemiReg}\NormDiscL{\DmFunc}{2}{\Dm}\NormDiscHz{\SolApp}{1}{\Dm}. 
\end{align*}
Consequently, we obtain \eqref{thm:stability:ineq}. 
\end{proof}

\begin{remark}
For function space $\FsSolApp^n\,(n\in\dN)$, the discrete $L^2$ norm, $\NormDiscL{\cdot}{2}{\Dm}$, satisfies the conditions of the norm. 
Moreover, if and only if particle distribution $\PtSet$ satisfies $\h$-connectivity, then the discrete $H^1_0$ norm, $\NormDiscHz{\cdot}{1}{\Dm}$, satisfies the conditions of the norm. 
\end{remark}

\section{Concluding remarks}
\label{sec6}
We have analyzed the stability of a generalized particle method for a Poisson equation with a source term given by a divergence form. 
We have obtained a unique solvability of the discretized Poisson equation by introducing a connectivity condition for particle distributions, which is referred to as $h$-connectivity.
Moreover, we have established the stability of the discretized Poisson equation based on the semi-regularity of a family of discrete parameters and discrete Sobolev norms with properties such as integration by parts. 

In future, we will analyze the error estimates of the discretized Poisson equation by showing properties such as the discrete Poincar\'{e} inequality. 
Moreover, we will extend these results to the incompressible viscous flow problem. 

\appendix
\section{Derivation of approximate operators}
\label{appendix:deriv_app_diver}
Assume a two-dimensional or three-dimensional space, $\Dim=2,3$. 
Assume a particle distribution on a square lattice with spacing $\Dx$. 
For $\i,\j=1,2,\dots,\N$, let $\DmLattice{\i}=(\Pt{\i}{}-\Dx/2,\Pt{\i}{}+\Dx/2)^\Dim$, $\BdLattice{\i}{\j}\deq\DmLatticeOl{\i}\cap\DmLatticeOl{\j}$, and $\IndexSetLatticeNext{\i}\deq\SetNd{\k=1,2,\dots,\N}{\Meas{\BdLattice{\i}{\k}}\neq0,\,\k\neq\i}$. 
As $\Meas{\DmLattice{\i}}=\Dx^\Dim$ and  $\Meas{\BdLattice{\i}{\j}}=\Dx^{\Dim-1}\,(\j\in\IndexSetLatticeNext{\i})$, by the divergence theorem, we can approximate the divergence of $\psi:\DmH\ra\dRd$ on $\Pt{\i}{}\in\PtSet$ as
\begin{align}
	\Div\psi_{\i}
		&\approx\dfrac{1}{\Meas{\DmLattice{\i}}}\int_{\DmLattice{\i}}\Div\psi(x)\dx\\
		&=\dfrac{1}{\Dx^\Dim}\int_{\partial\DmLattice{\i}}\psi(x)\cdot n\,ds\\
		&\approx\dfrac{1}{\Dx^\Dim}\sum_{\j\in\IndexSetLatticeNext{\i}}\Meas{\BdLattice{\i}{\j}}\psi\BrS{\dfrac{\Pt{\i}{}+\Pt{\j}{}}{2}}\cdot\dfrac{\Pt{\j}{}-\Pt{\i}{}}{|\Pt{\j}{}-\Pt{\i}{}|}\\
		 &\approx\dfrac{1}{\Dx^\Dim}\sum_{\j\in\IndexSetLatticeNext{\i}}\Meas{\BdLattice{\i}{\j}}\dfrac{\psi_{\j}+\psi_{\i}}{2}\cdot\dfrac{\Pt{\j}{}-\Pt{\i}{}}{|\Pt{\j}{}-\Pt{\i}{}|}\\
		 &=\dfrac{1}{2}\sum_{\j\in\IndexSetLatticeNext{\i}}\dfrac{\psi_{\j}+\psi_{\i}}{|\Pt{\j}{}-\Pt{\i}{}|}\cdot\dfrac{\Pt{\j}{}-\Pt{\i}{}}{|\Pt{\j}{}-\Pt{\i}{}|}, 
		 \label{appendix:approx_diver}
\end{align}
where $n$ is the outward normal vector on the boundary $\partial\DmLattice{\i}$. 
Further, using the central difference, we can approximate the Laplacian of $\phi:\DmH\ra\dR$ on $\Pt{\i}{}\in\PtSet$ as
\begin{align}
	\Lap\phi_{\i}
		&\approx\dfrac{1}{\Meas{\DmLattice{\i}}}\int_{\DmLattice{\i}}\Lap\phi(x)\dx\nn\\
		&=\dfrac{1}{\Dx^\Dim}\int_{\partial\DmLattice{\i}}\Grad\phi(x)\cdot n\,ds\\
		&\approx\dfrac{1}{\Dx^\Dim}\sum_{\j\in\IndexSetLatticeNext{\i}}\Meas{\BdLattice{\i}{\j}}\Grad\phi\BrS{\dfrac{\Pt{\i}{}+\Pt{\j}{}}{2}}\cdot\dfrac{\Pt{\j}{}-\Pt{\i}{}}{|\Pt{\j}{}-\Pt{\i}{}|}\\
		 &\approx\dfrac{1}{\Dx^\Dim}\sum_{\j\in\IndexSetLatticeNext{\i}}\Meas{\BdLattice{\i}{\j}}\dfrac{\phi_{\j}-\phi_{\i}}{|\Pt{\j}{}-\Pt{\i}{}|}\\
		 &=\sum_{\j\in\IndexSetLatticeNext{\i}}\dfrac{\phi_{\j}-\phi_{\i}}{|\Pt{\j}{}-\Pt{\i}{}|^2}
		 \label{appendix:approx_lap}
\end{align}
By noting that the number of elements of $\IndexSetLatticeNext{\i}$ is $2\Dim$, we can derive approximate operators \eqref{def:DivAppPlus} and \eqref{def:LapApp} as the weighted averages of \eqref{appendix:approx_diver} and \eqref{appendix:approx_lap}, respectively. 
As the approximation procedures in \eqref{appendix:approx_diver} and \eqref{appendix:approx_lap} are same as that of the finite volume method based on Voronoi decomposition, we can regard approximate operators \eqref{def:DivAppPlus} and \eqref{def:LapApp} as approximations based on the finite volume method. 

\section*{Acknowledgments}
This work was supported by JSPS KAKENHI Grant Number 17K17585 and by JSPS A3 Foresight Program.


\end{document}